\documentclass[12pt, reqno]{amsart}
\usepackage{amsmath, amsthm, amscd, amsfonts, amssymb, graphicx, color}
\usepackage[bookmarksnumbered, colorlinks, plainpages]{hyperref}
\input{mathrsfs.sty}

\newtheorem{theorem}{Theorem}[section]
\newtheorem{lemma}[theorem]{Lemma}
\newtheorem{proposition}[theorem]{Proposition}
\newtheorem{corollary}[theorem]{Corollary}
\theoremstyle{definition}
\newtheorem{definition}[theorem]{Definition}
\newtheorem{example}[theorem]{Example}

\theoremstyle{remark}

\numberwithin{equation}{section}

\begin{document}
\title{Non-linear positive maps between $C^*$-algebras}
\author[A. Dadkhah, M. S. Moslehian]{Ali Dadkhah and Mohammad Sal Moslehian}
\address{Department of Pure Mathematics, Center Of Excellence in Analysis on Algebraic Structures (CEAAS), Ferdowsi University of Mashhad, P. O. Box 1159, Mashhad 91775, Iran}

\email{dadkhah61@yahoo.com}
\email{moslehian@um.ac.ir}

\subjclass[2010]{15A60, 47A63.}
\keywords{$C^*$-algebra; $n$-positive map; Lieb map; superadditive; nonlinear map.}
\begin{abstract}
We present some properties of (not necessarily linear) positive maps between $C^*$-algebras. We first extend the notion of Lieb functions to that of Lieb positive maps between $C^*$-algebras. Then we give some basic properties and fundamental inequalities related to such maps. Next, we study $n$-positive maps ($n\geq 2$). We show that if for a unital $3$-positive map $\Phi: \mathscr{A}\longrightarrow\mathscr{B}$ between unital $C^*$-algebras and some $ A\in \mathscr{A}$ equality $\Phi(A^*A)= \Phi(A)^* \Phi(A)$ holds, then $\Phi(XA)=\Phi (X)\Phi (A)$ for all $X \in \mathscr{A}$.

In addition, we prove that for a certain class of unital positive maps $\Phi: \mathscr{A}\longrightarrow\mathscr{B}$ between unital $C^*$-algebras, the inequality $\Phi(\alpha A)\leq\alpha \Phi(A)$ holds for all $ \alpha \in [0,1]$ and all positive elements $ A\in \mathscr{A}$ if and only if $\Phi(0)=0$. Furthermore, we show that if for some $\alpha$ in the unit ball of $\mathbb{C}$ or in $\mathbb{R}_+$ with $|\alpha|\neq 0,1$, the equality $\Phi(\alpha I)=\alpha I$ holds, then $\Phi$ is additive on positive elements of $\mathscr{A}$. Moreover, we present a mild condition for a $6$-positive map, which ensures its linearity.
\end{abstract} \maketitle
\section{Introduction and preliminaries}

Let $\mathbb{B}(\mathscr{H})$ denote the $C^*$-algebra of all bounded linear operators on a complex Hilbert space $(\mathscr{H}, \langle\cdot,\cdot\rangle)$ with the unit $I$. If $\mathscr{H}=\mathbb{C}^n$, we identify $\mathbb{B}(\mathbb{C}^n)$ with the matrix algebra of $n\times n$ complex matrices $M_n(\mathbb{C})$. We consider the usual L\"{o}wner order $\leq$ on the real space of self-adjoint operators. An operator $A$ is said to be strictly positive (denoted by $ A >0$) if it is a positive invertible operator. According to the Gelfand--Naimark--Segal theorem, every $C^*$-algebra can be regarded as a $C^*$-subalgebra of $\mathbb{B}(\mathscr{H})$ for some Hilbert space $\mathscr{H}$. We use $\mathscr{A},\mathscr{B}, \cdots $ to denote $C^*$-algebras and $\mathscr{A}_+$ and $\mathscr{A}_{++}$ to denote the sets of positive and positive invertible elements of $\mathscr{A}$, respectively.
A map $\Phi: \mathscr{A}\to \mathscr{B}$ between $C^*$-algebras is said to be $*$-map or self-adjoint if $\Phi(A^*)=\Phi(A)^*$. It is positive if $\Phi(A)\geq 0$ whenever $A \geq 0$. It is called strictly positive if $A>0$ implies that $\Phi(A)>0$.
We say that $\Phi$ is unital if $\mathscr{A}, \mathscr{B}$ are unital and $\Phi$ preserves the unit. In this case, we simply denote both units of $\mathscr{A}$ and $\mathscr{B}$ by $I$. A map $\Phi$ is called $n$-positive if the map $\Phi_n: M_n(\mathscr{A})\to M_n(\mathscr{B})$ defined by $\Phi_n([a_{ij}]) = [\Phi(a_{ij})]$ is positive, where $M_n(\mathscr{A})$ stands for the $C^*$-algebra of $n \times n$ matrices with entries in $\mathscr{A}$. A map $\Phi$ is said to be completely positive if it is $n$-positive for all $n\in \mathbb{N}$; see \cite{AND, PAU, HAN}.
\\
It is known that for $A>0, B\geq 0$ as two elements in $ M_n(\mathscr{A})$,
\begin{eqnarray}\label{msm1}
\begin{bmatrix}A& X\\ X^*& B \end{bmatrix} \geq 0 \Longleftrightarrow B\geq X^* A^{-1} X
\end{eqnarray}
(see \cite[Lemma 2.1]{choi4}). Moreover, if $\mathscr{A}$ is a von Neumann algebra (that is a unital $*$-subalgebra of $\mathbb{B}(\mathscr{H})$ being closed in the strong operator topology), then the matrix $\begin{bmatrix}A& X\\ X^*& B \end{bmatrix}$ is positive if and only if there exists $K\in \mathscr{A}$ with $\|K\|\leq 1$ (called a contraction) such that $X= A^{\frac{1}{2}}K B^\frac{1}{2}$; see \cite[Proposition 2.5]{balkan}. \\ A map $\Phi: \mathscr{X}\subseteq \mathscr{A}\to \mathscr{B}$ is called superadditive on a subset $\mathscr{X}$ of $\mathscr{A}$, which is closed under addition, if
\begin{eqnarray*}
\Phi(A+B)\geq \Phi(A)+\Phi(B)
\end{eqnarray*}
for every $A,B \in \mathscr{X}$ and it is called strongly superadditive, if
\begin{eqnarray}\label{ssupa}
\Phi(A+B+C)+\Phi(A)\geq \Phi(A+B)+\Phi(A+C)
\end{eqnarray}
for every $A,B,C \in \mathscr{X}$. A map $\Phi: \mathscr{X}\subseteq \mathscr{A}\to \mathscr{B}$ is called starshaped if $\Phi(\alpha A)\leq \alpha \Phi(A)$ for every $A\in \mathscr{X}$ and $\alpha \in [0,1]$. It is known that every starshaped function $f: [0,\infty) \to [0,\infty)$ is superadditive (see \cite[Theorem 5]{bruk2}). However the converse of this statement is not true, in general (see \cite[ page 422]{beck}).

For example, the usual determinant and the permanent are completely positive functions on $M_n(\mathbb{C})$ as well as strongly superadditive functions over the cone of positive semidefinite matrices in $M_n(\mathbb{C})$ (see \cite{choi} and \cite{li}).

Recently, Paksoy et al. \cite{paksoy} gave an extension of inequality (\ref{ssupa}) for
generalized matrix functions of positive semidefinite matrices associated with any subgroup of the
permutation group and Lin and Sra \cite{lin} generalized this inequality for self-adjoint matrices.

Ando and Choi \cite{choi} characterized completely positive non-linear maps between $C^*$-algebras. They showed that such maps can be represented as the sum of completely positive maps which are mixed homogeneous. Moreover, they showed that \cite[Theorem 3]{choi} if $\Phi:\mathscr{A}\to \mathscr{B}$ is a completely positive map between $C^*$-algebras, then it is strongly superadditive on positive elements. In addition, they showed that a non-linear completely positive map $\Phi:\mathscr{A}\to \mathscr{B}$
 is additive if and only if $\Phi(2A)=2\Phi(A)$ for each $A\in \mathscr{A}$.
Choi \cite[Theorem 3.1 and Corollary 2.8]{choi3} showed that if $\Phi:\mathscr{A}\to \mathscr{B}$ is a unital $2$-positive
linear map between unital $C^*$-algebras, then $\Phi(A^*A)\geq \Phi(A^*)\Phi(A)$ for every
 $A \in \mathscr{A}$ and if for some $A$, equality $\Phi(A^*A)=\Phi(A^*) \Phi(A)$ holds, then
\begin{eqnarray*}
\Phi(XA)=\Phi(X)\Phi(A)
\end{eqnarray*}
for every $X \in \mathscr{A}$.

G\"{u}nther and Klotz \cite{gunt} investigated $n$-positive norms on $M_k(\mathbb{C})$ as
$n$-positive maps. They show that there is not any $4$-positive norm on
$M_k(\mathbb{C})$.

The aim of this paper is to extend some properties of linear positive maps for general positive maps. In the next section, we give the notion of Lieb map as an extension of Lieb function. We then prove some basic properties of such maps and give a relationship between Lieb maps and $2$-positive maps.

Section 3 is devoted to the study of 3-positive maps. We show
that for every unital $3$-positive map $\Phi:\mathscr{A}\to \mathscr{B}$ between $C^*$-algebras the validity of $\Phi(A^*A)= \Phi(A)^*\Phi(A)$ for some $A\in \mathscr{A}$ implies that $\Phi(XA) =
\Phi(X)\Phi(A)$ for all $X\in \mathscr{A}$. Furthermore, we show that
this result does not hold for non-linear $2$-positive maps, in general. This gives an extension of a result due to Choi for $2$-positive linear maps. We then prove that every $3$-positive map $\Phi:\mathscr{A}\to \mathscr{B}$ is superadditive on positive elements whenever $\Phi(0) = 0$, and yield some new results on the $n$-positivity of any $C^*$-norm.

Section 4 begins with a study of a class of positive maps by establishing that for a certain class of positive maps $\Phi:\mathscr{A}\to \mathscr{B}$, the inequality $\Phi(\alpha A)\leq\alpha \Phi(A)$ holds for all $ \alpha \in [0,1]$ and all positive elements $ A\in \mathscr{A}$ if and only if $\Phi(0)=0$. Furthermore, we show that if for some $\alpha$ in the unit ball of $\mathbb{C}$ or in $\mathbb{R}_+$ with $|\alpha|\neq 0,1$, the equality $\Phi(\alpha I)=\alpha I$ holds, then $\Phi$ is additive on positive elements of $\mathscr{A}$. Finally, we prove that for
some classes of positive maps $\Phi:\mathscr{A}\to \mathscr{B}$, it is linear if and only if $\Phi\big((r+\mbox{i}s) I\big)= (r+\mbox{i}s) I$ for some $r,s \in \mathbb{R}$ with $0 \neq|r+\mbox{i}s|<1 $ and $s\neq 0$.

\section{Some basic properties of Lieb maps }

In this section, we extend the definition of a Lieb function \cite{lieb} by introducing the concept of Lieb map between $C^*$-algebras. Then we study some significant properties of such maps. \\
Recall that a map $\varphi: \mathscr{A}\to \mathbb{C}$ is called a Lieb function if it has the following properties:
\begin{itemize}
\item(Monotonicity) $\varphi(A) \geq \varphi (B)\geq 0 $ if $A\geq B \geq 0$,
\item (Cauchy--Schwarz) $\varphi(A^*A)\varphi(B^*B)\geq |\varphi(A^*B)|^2 $ for every $A,B \in \mathscr{A}$.
\end{itemize}
We generalize this definition as follows:
\begin{definition}\label{lieb}
Let $\mathscr{A}, \mathscr{B}$ be two $C^*$-algebras. A map $\Phi: \mathscr{A}\to \mathscr{B}$ is called a Lieb map if it has the following properties:
\begin{itemize}
\item(Monotonicity) $\Phi(A) \geq \Phi(B)$ if $A\geq B \in \mathscr{A}_{+}$,
\item (Cauchy--Schwarz) $\begin{bmatrix}
\Phi(A^*A) & \Phi(A^*B) \\ \Phi(A^*B)^* & \Phi(B^*B)
\end{bmatrix}\geq 0 $ for every $A,B \in \mathscr{A}$.
\end{itemize}
\end{definition}
The monotonicity ensures that all Lieb maps are positive (see the proof of Proposition \ref{matrixeq}). Note that a Lieb map is not continuous, in general. There are several significant examples of Lieb maps.
\begin{example}\noindent
\begin{enumerate}
\item ٍEvery Lieb function $\varphi:\mathscr{A}\to \mathbb{C}$.
\item Every 2-positive map $\Phi: \mathscr{A}\to \mathscr{B}$; see Corollary \ref{cor1} (1).
\item The norm of any $C^*$-algebra is $2$-positive and so is a Lieb map; see Corollary \ref{norm2}.
\item The determinant and the usual trace on $M_n(\mathbb{C})$ are completely positive and so are Lieb maps.
\end{enumerate}
\end{example}
We aim to give an equivalent definition of a Lieb map via positive block matrices. The following lemma is an easy consequence of \cite[Theorem 1.3.3]{bhatia}.

\begin{lemma}\label{epsilpo}
Let $A,B \geq 0$. Then matrix $\begin{bmatrix}
A & X \\ X^* & B
\end{bmatrix}\geq 0 $ if and only if $B\geq X^* (A+\varepsilon I)^{-1} X $ for every $\varepsilon >0$.
\end{lemma}

Our first result reads as follows.

\begin{proposition}\label{matrixeq}
Let $\mathscr{A}$ be a von Neumann algebra and $\mathscr{B}$ be a unital $C^*$-algebra. Then a map $\Phi: \mathscr{A}\to \mathscr{B}$ is Lieb if and only if
\begin{eqnarray}\label{postmat}
\begin{pmatrix} \Phi(A) &\Phi(C) \\ \Phi(C)^* & \Phi(B) \end{pmatrix}\geq 0
\end{eqnarray}
for every $A,B, C \in \mathscr{A}$ such that $\begin{pmatrix} A &C \\ C^* & B \end{pmatrix}\geq 0$.
\end{proposition}
\begin{proof}
($\Longrightarrow$) Let $\begin{pmatrix} A &C \\ C^* & B \end{pmatrix}\geq 0$. Therefore, there exists a contraction $K$ such that $C=A^{\frac{1}{2}} K B^{\frac{1}{2}}$. Since $K$ is a contraction, we get $K^*K\leq I$ and therefore $B^{\frac{1}{2}}K^*K B^{\frac{1}{2}}\leq B$. The monotonicity of $\Phi$ therefore implies $\Phi(B^{\frac{1}{2}}K^*K B^{\frac{1}{2}})\leq \Phi(B)$, and hence
\begin{eqnarray*}
\begin{pmatrix} \Phi(A) &\Phi(C) \\ \Phi(C)^* & \Phi(B) \end{pmatrix}\geq \begin{pmatrix} \Phi(A) &\Phi(A^{\frac{1}{2}}K B^{\frac{1}{2}}) \\ \Phi(A^{\frac{1}{2}} K B^{\frac{1}{2}})^* & \Phi(B^{\frac{1}{2}}K^*K B^{\frac{1}{2}}) \end{pmatrix},
\end{eqnarray*}
and the last matrix is positive by the Cauchy--Schwarz part of the definition of a Lieb map.

($\Longleftarrow$) First note that $\Phi$ is a positive map. To see this, take $B=C=0$ in (\ref{postmat}). We first observe that for positive elements $X,Y$ in a $C^*$-algebra, we have
\begin{eqnarray} \label{positivee}
X \geq Y {\rm\ if\ and\ only\ if\ } \begin{pmatrix} X&Y \\ Y & Y \end{pmatrix}\geq 0.
\end{eqnarray}
Indeed, for arbitrary $\varepsilon>0$ we can write
\begin{eqnarray*}
 \begin{pmatrix} X&Y \\ Y & Y \end{pmatrix} + \begin{pmatrix} \varepsilon I &\varepsilon I \\ \varepsilon I & \varepsilon I \end{pmatrix} \geq 0 &\Longleftrightarrow& \begin{pmatrix} X+\varepsilon I &Y+\varepsilon I \\ Y+\varepsilon I & Y+\varepsilon I \end{pmatrix} \geq 0\\ & \Longleftrightarrow& X+\varepsilon I \geq (Y+\varepsilon I) (Y+\varepsilon I)^{-1} (Y+\varepsilon I) \\ & & \qquad \qquad \qquad \text{(by \eqref{msm1})}\\ &\Longleftrightarrow& X+\varepsilon I \geq Y+\varepsilon I \\ & \Longleftrightarrow& X \geq Y.
\end{eqnarray*}
Hence, by letting $\varepsilon \longrightarrow 0$, we get $X\geq Y \Longleftrightarrow \begin{pmatrix} X&Y \\ Y & Y \end{pmatrix}\geq 0$.
 \\
 Let $A\geq B \geq 0$. Utilizing (\ref{positivee}), we get $\begin{pmatrix} A &B \\ B & B \end{pmatrix}\geq 0$. Applying the assumption, we get $\begin{pmatrix} \Phi(A) & \Phi(B) \\ \Phi(B) & \Phi(B) \end{pmatrix}\geq 0$. Using (\ref{positivee}) again, we conclude that $ \Phi(B) \leq \Phi(A)$, which proves the monotonicity condition for $\Phi$.
 The Cauchy--Schwarz inequality immediately follows from the positivity of matrix $
\begin{pmatrix} A^*A & A^*B \\ B^*A & B^*B \end{pmatrix}
$.
\end{proof}
\begin{corollary}\label{cor1}\noindent
\begin{enumerate}
\item If $\Phi$ is a $2$-positive map between unital $C^*$-algebras, then $\Phi$ is a Lieb map.
\item Let $\mathscr{A}$ be a von Neumann algebra and $\mathscr{B}$ be a unital $C^*$-algebra. Let $\Phi: \mathscr{A}\to \mathscr{B}$ be a $*$-map. Then $\Phi$ is Lieb map if and only if $\Phi$ is $2$-positive.
\end{enumerate}
\end{corollary}
\begin{proof}
(1) According to the proof of direction ($\Longleftarrow$) in the proposition \ref{matrixeq}, we only need to prove that every $2$-positive map is self-adjoint. It is known that matrix $\begin{pmatrix} |A| &A^* \\ A & |A^*| \end{pmatrix}$ is positive for every $A\in \mathscr{A}$ (see \cite[Theorem 2]{choi}). Using the $2$-positivity of $\Phi$, we see that the matrix
\begin{eqnarray*}
\begin{pmatrix} \Phi(|A|) &\Phi(A^*) \\ \Phi(A) & \Phi(|A^*|) \end{pmatrix}
\end{eqnarray*}
is positive and so self-adjoint. This implies that $\Phi(A^*)^*=\Phi(A)$.\\
(2) It immediately follows from Proposition \ref{matrixeq}.\\
\end{proof}
There is a close relationship between Lieb maps and $2$-positive maps. However, there are Lieb maps, which are not self-adjoint. For example, let $\Phi$ be a $2$-positive map such that for a non-selfadjoint element $A$, $\Phi(A)\neq 0$. Then the map $\Psi$ defined by
\[ \Psi(X) =
 \begin{cases}
 \Phi(X) & \quad X\neq A\\
 0 & \quad X=A
 \end{cases}
\]
is a Lieb map (by an easy application of Proposition \ref{matrixeq}), which is not a $*$-map since, by the proof of Corollary \ref{cor1}, each $2$-positive map is a $*$-map.\\
It is easy to show that a positive linear map $\Phi$ between unital $C^*$-algebras is strictly positive if and only if $\Phi(I)>0$. We claim that this fact holds for all Lieb maps between unital $C^*$-algebras. Indeed, if $A>0$ there is $ \varepsilon \in (0,1)$ such that $A\geq \varepsilon I$. Since $\Phi$ is monotone, we get $\Phi(A)\geq \Phi(\varepsilon I)$. We only need to show that $\Phi(\varepsilon I) $ is invertible.\\ First assume that $\Phi(I)=I$. The positivity and the Cauchy--Schwarz property of $\Phi$ implies that
\begin{eqnarray*}
\begin{bmatrix} \Phi(\varepsilon I) & \Phi(I) \\ \Phi(I) &\Phi(\varepsilon^{-1} I)
\end{bmatrix}= \begin{bmatrix} \Phi(| \sqrt{\varepsilon}I|^2) & \Phi(\sqrt{\varepsilon}\sqrt{\varepsilon^{-1}}I) \\ \Phi(\sqrt{\varepsilon}\sqrt{\varepsilon^{-1}}I)^* &\Phi(| \sqrt{\varepsilon^{-1}}I|^2)
\end{bmatrix}\geq 0.
\end{eqnarray*}
This shows that
\begin{eqnarray*}
\Phi(\varepsilon I)\geq \Phi(I) \Phi(\varepsilon^{-1}I)^{-1} \Phi(I)=\Phi(\varepsilon^{-1}I)^{-1}\geq \|\Phi(\varepsilon^{-1}I)\|^{-1}I,
\end{eqnarray*}
which implies that $\Phi(\varepsilon I)$ is invertible. Note that $\Phi(\varepsilon^{-1}I)\geq \Phi(I)=I$ is invertible.

In the case when $\Phi$ is not unital, define map $\Psi(X)=\Phi(I)^{-\frac{1}{2}} \Phi(X) \Phi(I)^{-\frac{1}{2}}$. It is easy to check that $\Psi$ is a unital Lieb map. By the first part of the proof, there exists $\alpha >0$ such that $\Psi(\varepsilon I) \geq \alpha I$. Since $\Phi(I)>0$, there exists a $\beta>0$ such that $\Phi(I) \geq \beta I$. Consequently, $\Phi(\varepsilon I)= \Phi(I)^\frac{1}{2} \Psi(\varepsilon I) \Phi(I)^\frac{1}{2}\geq \alpha \beta I>0$. It follows from Corollary \ref{cor1} (1) that the above fact also holds for a $2$-positive map between unital $C^*$-algebras.\\

\begin{corollary}[Choi's inequalities]\label{Choi1}
Let $\mathscr{A}$ and $\mathscr{B}$ be unital $C^*$-algebras. If $\Phi: \mathscr{A}\to \mathscr{B}$ is a Lieb map such that $\Phi(I)>0$, then the following statements hold:
\begin{enumerate}
\item[(a)] $ \Phi(|A|^2)\geq \Phi(A)^*\Phi(I)^{-1} \Phi(A)$ for every $A\in \mathscr{A}$.
In particular, if \newline $0< \Phi(I) \leq I$, then $ \Phi(|A|^2)\geq |\Phi(A)|^2$.
\item[(b)] If $\Phi$ is unital, then $\Phi(A^{-1}) \geq \Phi(A)^{-1}$ for every $A\in \mathscr{A}_{++}$.
\end{enumerate}
Furthermore, if $\Phi: \mathscr{A}\to \mathscr{B}$ is a $2$-positive map between unital $C^*$-algebras, then \text{\rm (a)} and \text{\rm (b)} are still true.
\end{corollary}
\begin{proof}
(a) The positivity and the Cauchy--Schwarz property of $\Phi$ implies that
\begin{eqnarray*}
\begin{pmatrix}\Phi(I) & \Phi(A) \\ \Phi(A)^*&\Phi(A^*A) \end{pmatrix}= \begin{pmatrix}\Phi(I^*I) & \Phi(I^*A) \\ \Phi(I^*A)^*&\Phi(A^*A) \end{pmatrix}\geq 0,
\end{eqnarray*}
which ensures that
\begin{eqnarray*}
 \Phi(A^*A) \geq {\Phi}(A)^* \Phi(I)^{-1} {\Phi}(A).
\end{eqnarray*}
Second inequality follows from the fact that $\Phi(I)^{-1} \geq I$ whenever $0< \Phi(I) \leq I$.\\
(b) According to the positivity and the Cauchy--Schwarz property of $\Phi$, we get
\begin{eqnarray*}
\begin{pmatrix}\Phi(A) & \Phi(I) \\ \Phi(I)^*&\Phi(A^{-1}) \end{pmatrix}=\begin{pmatrix}\Phi(|A^\frac{1}{2}|^2) & \Phi(A^\frac{1}{2}A^{-\frac{1}{2}}) \\ \Phi(A^\frac{1}{2}A^{-\frac{1}{2}})^*&\Phi(|A^{-\frac{1}{2}}|^2) \end{pmatrix}\geq 0
\end{eqnarray*}
 for every $A>0$. Hence,
\begin{eqnarray*}
 \Phi(A^{-1}) \geq {\Phi}(I) \Phi(A)^{-1} {\Phi}(I)= \Phi(A)^{-1}.
\end{eqnarray*}
The statement about $2$-positive maps now follows directly from the statement about Lieb maps and Corollary \ref{cor1} (1).
\end{proof}
The next corollary is another consequence of Proposition \ref{matrixeq}.
\begin{corollary}\label{norm2}
Let $\mathscr{A}$ be a von Neumann algebra. Then an algebraic norm $\|\cdot \|$ on $\mathscr{A}$ is Lieb if and only if it is the $C^*$-norm on $\mathscr{A}$.
\end{corollary}
\begin{proof}
Suppose that $\|\cdot \|$ is a Lieb map. According to the positivity of the matrix $\begin{pmatrix} I & A \\ A^* &A^*A \end{pmatrix}$, we have
\begin{eqnarray*}
 \begin{pmatrix}\| I\| & \|A\| \\ \|A\|^* &\|A^*A\| \end{pmatrix}\geq 0,
\end{eqnarray*}
whence we get
\begin{eqnarray*}
\|A^*A\| \geq \|A\|^2,
\end{eqnarray*}
which implies that $\|\cdot \|$ is a $C^*$-norm.\\ Conversely, if $\begin{pmatrix} A & X^* \\ X&B \end{pmatrix}\geq 0$, then by Lemma \ref{epsilpo}, we have
\begin{eqnarray*}
A\geq X^* (B+\varepsilon I)^{-1}X,
\end{eqnarray*}
for arbitrary small $\varepsilon >0$. The inequality $(B+\varepsilon I)^{-1} \geq \|B+\varepsilon I\|^{-1}I$ implies that
\begin{eqnarray*}
A\geq X^* \|B+\varepsilon I\|^{-1}X,
\end{eqnarray*}
whence $\|A\|\|B+\varepsilon I\| \geq \|X^*X\|=\|X^*\|\|X\|$, which ensures that 
\[\begin{pmatrix}\| A\| & \|X^*\| \\ \|X\| &\|B+\varepsilon I\| \end{pmatrix}\geq 0.\]
Hence,
$ \begin{pmatrix}\| A\| & \|X^*\| \\ \|X\| &\|B\| \end{pmatrix}\geq 0$.
\end{proof}
We remark that Corollary \ref{norm2} also holds if $\|\cdot \|$ is a $2$-positive norm on a unital $C^*$-algebra.

\section{$3$-positivity of non-linear maps }

This section gives some properties of $3$-positive maps. We begin by proving a property of $2$-positive linear maps for non-linear $3$-positive maps. \\
The following lemma is straightforward, so we omit its proof.
\begin{lemma}\label{kelidi}
Let $A\geq 0$ and $X$ be elements in a unital $C^*$-algebras. Then $\begin{bmatrix} 0 & X \\ X^* & A
\end{bmatrix}\geq 0$ if and only if $X=0$.
\end{lemma}

Choi \cite[Theorem 3.1]{choi3} showed that if $\Phi: \mathscr{A}\to \mathscr{B}$ is a unital $2$-positive linear map between unital $C^*$-algebras such that $\Phi(A^*A)= \Phi(A^*)\Phi(A)$ for some $A\in \mathscr{A}$, then $\Phi(XA)= \Phi(X)\Phi(A)$ for every $X\in \mathscr{A}$. We remark that this result does not hold for an arbitrary unital $2$-positive map. Indeed, for the $C^*$-norm of a $C^*$-algebra as a unital $2$-positive map (see Corollary \ref{norm2}), $\|A^*A\|=\|A^*\| \|A\|$ for every $A\in \mathscr{A}$. However, the equality $\|AX\|=\|A\| \|X\|$ is not valid, in general.

The following theorem gives a generalization of the above fact for non-linear $3$-positive maps.

\begin{theorem} \label{3positive}
Suppose that $\Phi: \mathscr{A}\to \mathscr{B}$ is a unital $3$-positive map. If for some $A\in \mathscr{A}$ equality $\Phi(A^*A)= \Phi(A)^*\Phi(A)$ holds, then
\begin{eqnarray*}
 \Phi(A^*X)=\Phi (A^*)\Phi (X){\text{\ \ and\ \ }} \Phi(XA)=\Phi (X)\Phi (A)
\end{eqnarray*}
for every $X \in \mathscr{A}$.
\end{theorem}
\begin{proof}
Suppose that $A\in \mathscr{A}$ satisfies $\Phi(A^*A)= \Phi(A)^*\Phi(A)$. Let $X\in \mathscr{A}$. We have
\begin{eqnarray*}
\begin{bmatrix} A^*A & A^*X & A^* \\ X^*A & X^*X & X^* \\ A& X & I
\end{bmatrix} =\begin{bmatrix} A^*&0&0 \\ X^*&0&0 \\ I&0&0 \end{bmatrix}\begin{bmatrix} A&X&I\\ 0&0&0 \\ 0&0&0 \end{bmatrix}\geq 0.
\end{eqnarray*}
Since $\Phi$ is $3$-positive, we have
\begin{eqnarray}\label{marix3p}
\begin{bmatrix} \Phi(A^*A) & \Phi(A^*X) & \Phi(A^*) & 0 \\ \Phi(X^*A) & \Phi(X^*X) & \Phi(X^*) & 0\\ \Phi(A)& \Phi(X)& \Phi(I) &0 \\ 0 & 0 & 0 & \Phi(I)
\end{bmatrix}\geq 0.
\end{eqnarray}
Using the positivity of matrix (\ref{marix3p}), we deduce that
\begin{eqnarray*}
\begin{bmatrix} \Phi(A^*A) & \Phi(A^*X) \\ \Phi(X^*A) & \Phi(X^*X)
\end{bmatrix} &\geq & \begin{bmatrix} \Phi(A^*) & 0 \\ \Phi(X^*) & 0\end{bmatrix} \begin{bmatrix} \Phi(I) & 0\\ 0 & \Phi(I)
\end{bmatrix} ^{-1} \begin{bmatrix} \Phi(A) & \Phi(X) \\ 0 & 0\end{bmatrix} \\ &=& \begin{bmatrix} \Phi(A^*)\Phi(A) & \Phi(A^*)\Phi(X) \\ \Phi(X^*)\Phi(A) & \Phi(X^*)\Phi(X)
\end{bmatrix}.
 \end{eqnarray*}
Since $\Phi(A^*A)= \Phi(A)^*\Phi(A)$ and $\Phi$ is self-adjoint, we conclude that
\begin{eqnarray*}
0&\leq& \begin{bmatrix} \Phi (A^*A) & \Phi(A^*X) \\ \Phi (X^*A) & \Phi (X^*X)
\end{bmatrix} - \begin{bmatrix} \Phi (A^*)\Phi (A) & \Phi (A^*)\Phi (X) \\ \Phi (X^*)\Phi (A) & \Phi (X^*)\Phi (X)
\end{bmatrix}\\&=& \begin{bmatrix} 0& \Phi (A^*X)-\Phi (A^*)\Phi (X) \\ \Phi (X^*A)- \Phi (X^*)\Phi (A) &\Phi (X^*X)- \Phi (X^*)\Phi (X)
\end{bmatrix}.
 \end{eqnarray*}
 Since $\Phi(X^*X)-\Phi(X^*)\Phi(X) \geq 0$, by using Lemma \ref{kelidi}, we get $\Phi (A^*X)=\Phi (A^*)\Phi (X)$. Due to $\Phi$ is a $*$-map, $\Phi (XA)=\Phi (X)\Phi (A) $.
\end{proof}

\begin{corollary}
Let $\| \cdot\|$ be an algebraic norm on a unital $C^*$-algebra $\mathscr{A}$. If $\|\cdot \|$ is $3$-positive, then
\begin{eqnarray*}
\| XY\|=\|X\| \|Y\|
\end{eqnarray*}
for every $X,Y \in \mathscr{A}$.
\end{corollary}
\begin{example}
The map $|\cdot|$ is a $3$-positive norm on $\mathbb{C}$; see \cite[Section 7.5 Problem 4]{horn2}. To see an example of a $3$-positive norm on $M_n(\mathbb{C})$ we refer the reader to \cite[Proposition 6.2]{gunt}.
\end{example}
The next theorem, inspired by \cite[Theorem 3]{choi}, gives a result about the superadditivity of $3$-positive maps.
\begin{theorem}\label{3positive2}
Let $\Phi: \mathscr{A} \to \mathscr{B}$ be a $3$-positive map between unital $C^*$-algebras. Then
\begin{eqnarray}\label{alisup}
2\Phi(0)+ \Phi(A+B)\geq \Phi(A)+\Phi(B)
 \end{eqnarray}
for every positive elements $A,B \in \mathscr{A}$. In particular,
\begin{eqnarray*}
 \Phi(0)=0 \Longleftrightarrow \Phi(A+B)\geq \Phi(A)+\Phi(B).
\end{eqnarray*}
\end{theorem}
\begin{proof}
Let $A, B$ be positive. It follows from $\begin{bmatrix} A &A &0 \\ A & A & 0 \\ 0 & 0 &0 \end{bmatrix}\geq0 $ and $\begin{bmatrix} 0 &0 &0 \\ 0 & B & B \\ 0 & B &B \end{bmatrix}\geq0 $ that
\begin{eqnarray*}
\begin{bmatrix} A &A &0 \\ A & A+B & B \\ 0 & B &B \end{bmatrix}\geq0\,.
\end{eqnarray*}
Since $\Phi$ is $3$-positive, we have
\begin{eqnarray*}
 \begin{bmatrix} \Phi(A) &\Phi(A) &\Phi(0) \\ \Phi(A) & \Phi(A+B) & \Phi(B) \\ \Phi(0) & \Phi(B) &\Phi(B) \end{bmatrix} \geq 0.
\end{eqnarray*}
Therefore,
\begin{eqnarray*}
 0 &\leq& \begin{bmatrix} -1 & 1 & -1 \\ 0 & 0 & 0 \\ 0 & 0 &0 \end{bmatrix} \begin{bmatrix} \Phi(A) &\Phi(A) &\Phi(0) \\ \Phi(A) & \Phi(A+B) & \Phi(B) \\ \Phi(0) & \Phi(B) &\Phi(B) \end{bmatrix} \begin{bmatrix} -1 & 0 & 0 \\ \ \ 1 &0 & 0 \\ -1 & 0 & 0 \end{bmatrix}\\ &=&2 \Phi(0)+\Phi(A+B)-\Phi(A)-\Phi(B),
 \end{eqnarray*}
which ensures that $2 \Phi(0)+\Phi(A+B)\geq \Phi(A)+\Phi(B) $. \\
The rest is evident.
\end{proof}
Theorems \ref{3positive} and \ref{3positive2} show that the assumption of $3$-positivity is a strong condition on the norm of a $C^*$-algebra. Recall that a norm $\|\cdot\|$ on a normed space $\mathcal{X}$ is called strictly convex if for all non-zero elements $X,Y\in \mathcal{X}$, $\|X+Y\|=\|X\|+\|Y\|$ implies $X=\lambda Y$ for some $\lambda \in \mathbb{C}$.
\begin{corollary}
If $(\mathscr{A},\|\cdot \|)$ is a unital $C^*$-algebra and $\|\cdot \|$ is a strictly convex $3$-positive map, then
 $\mathscr{A}=\mathbb{C}$.
\end{corollary}
\begin{proof}
Assume that $P$ is a positive element of $\mathscr{A}$. Employing Theorem \ref{3positive2} we get $\|P\|+\|I\|=\|P+I\|$. Since $\|\cdot\|$ is strictly convex, we have $P=cI$ for some positive number $c$. The assertion follows from this and the fact that positive elements of $\mathscr{A}$ linearly generate it.

\end{proof}

\section{ non-linear maps of class $S_{{\rm mon} +}^{(n)}$ }

In this section, we study the strong superadditivity and the starshapness of maps in some classes of positive maps. These classes are between the classes of $n$-positive and $2n$-positive maps and enjoy of some interesting properties. Ando et. al \cite[Theorem 3]{choi} showed that if $\Phi$ is a $4$-positive map, then $\Phi$ is strongly superadditive on positive elements.
(In that theorem, they assumed that $\Phi$ is completely positive, however, in the proof they used only the $4$-positivity of $\Phi$.) \\
We say a positive map $\Phi: \mathscr{A} \to \mathscr{B}$ is in the class $S_{{\rm mon} +}^{(n)}$ whenever the map $\Phi_n$ is monotone on positive elements of $M_n(\mathscr{A})$. It is easy to show that every $2n$-positive map is in the class $S_{{\rm mon} +}^{(n)}$. Indeed, if $[A_{ij}]$ and $[B_{ij}]$ are two positive elements in $M_n(\mathscr{A})$ such that $[A_{ij}]\geq[B_{ij}]$, then we have $\begin{bmatrix} [A_{ij}] &[B_{ij}] \\ [B_{ij}] &[B_{ij}]
\end{bmatrix}\geq 0$ (see (\ref{positivee})). Moreover, the $2n$-positivity of $\Phi$ implies the $2$-positivity of $\Phi_n$. Hence,
\begin{eqnarray*}
\begin{bmatrix} \Phi_n\big([A_{ij}]\big) & \Phi_n\big([B_{ij}]\big) \\ \Phi_n\big( [B_{ij}]\big) &\Phi_n\big( [B_{ij}]\big)
\end{bmatrix}=\begin{bmatrix} [\Phi(A_{ij})] & [\Phi(B_{ij})] \\ [\Phi(B_{ij})] & [\Phi(B_{ij})]
\end{bmatrix} \geq 0,
\end{eqnarray*}
which implies that $ \Phi_n\big([A_{ij}]\big)\geq\Phi_n\big([B_{ij}]\big) $. It is however unknown to the authors whether every $3$-positive map is in the class $S_{{\rm mon} +}^{(2)}$.\\ There are several examples of maps in the class $S_{{\rm mon} +}^{(2)}$, which are not $4$-positive. For example, (i) power functions $\Phi_p: \mathbb{C} \to \mathbb{C}$ defined by $\Phi_p(x)=|x|^p\,\,(1<p<2)$ are in the class $S_{{\rm mon} +}^{(2)}$ (see \cite[Theorem 5.1]{hiai}), however, $\Phi_p$ $(1<p<2)$ are only $3$-positive but not $4$-positive (see \cite[Theorem 6.3.9]{horn}). Note that \cite[Theorem 6.3.9]{horn} states that the Hadamard power $[a_{ij}^p]$ of a positive semidefinite $3\times 3$ matrix $[a_{ij}]$ with nonnegative entries is positive semidefinite for $p\geq 1$. Hence, according to \cite[ Section 7.5 Problem 4]{horn2}, positive semidefiniteness of a $3\times 3$ matrix $[a_{ij}]$ implies positive semidefiniteness of $[|a_{i,j}|]$, which ensures $\Phi_p$ $(1<p<2)$ is $3$-positive; (ii) every positive semidefinite matrix $P\in M_n(\mathbb{C})$ induces a map $\phi_{P}: M_n(\mathbb{C}) \to \mathbb{C}$ defined by $\phi_{P}(A)=|{\rm tr}(AP)|$ such that it is a $3$-positive semi--norm on $M_n(\mathbb{C})$ (see \cite[Lemma 6.1]{gunt}). We show that $\phi_P$ is in the class $S_{{\rm mon} +}^{(2)}$. If $\begin{bmatrix} A & C \\ C^* &B \end{bmatrix}\geq \begin{bmatrix} A'& C'\\ C'^* &B' \end{bmatrix}\geq 0 $ are in $M_{2n}(\mathbb{C})$, then we have $ \begin{bmatrix} P^{\frac{1}{2}}A P^{\frac{1}{2}}& P^{\frac{1}{2}}CP^{\frac{1}{2}} \\ P^{\frac{1}{2}}C^* P^{\frac{1}{2}}&P^{\frac{1}{2}}BP^{\frac{1}{2}} \end{bmatrix}\geq \begin{bmatrix} P^{\frac{1}{2}}A'P^{\frac{1}{2}} & P^{\frac{1}{2}}C'P^{\frac{1}{2}} \\ P^{\frac{1}{2}} C'^*P^{\frac{1}{2}} &P^{\frac{1}{2}}B'P^{\frac{1}{2}} \end{bmatrix}$. Since the trace is completely positive, we get $ \begin{bmatrix} {\rm tr}(AP)& {\rm tr}( C P) \\ {\rm tr}( C P)^*& {\rm tr}( BP) \end{bmatrix}\geq \begin{bmatrix} {\rm tr}(A'P) & {\rm tr}( C' P) \\ {\rm tr}( C' P)^* & {\rm tr}(B' P) \end{bmatrix}$. This shows $\phi_{P}$ is in the class $S_{{\rm mon} +}^{(2)}$, since the map $|\cdot |$ on $\mathbb{C}$ is in the class $S_{{\rm mon} +}^{(2)}$ (see Example \ref{exfin}). However, $\phi_{P}$ can not be $4$-positive (see \cite[Corollary 2.3]{gunt}); (iii) It is easy to check that every positive linear functional $\varphi :\mathscr{A}\longrightarrow \mathbb{C}$ on a $C^*$-algebra induces a non-linear $3$-positive map $\Phi:\mathscr{A}\longrightarrow \mathbb{C}$ that belongs to the class $S_{{\rm mon} +}^{(2)}$ given by $\Phi(A)=|\varphi(A)|$, since every positive linear functional is completely positive and the map $|\cdot |$ on $\mathbb{C}$ is $3$-positive and in the class $S_{{\rm mon} +}^{(2)}$.

We now show that if $\Phi: \mathscr{A} \to \mathscr{B}$ is in the class $S_{{\rm mon} +}^{(2)}$, then it is strongly superadditive on positive elements of $\mathscr{A}$.\\
 Let $A,B,C \in \mathscr{A}_+$. Clearly,
\begin{eqnarray*}
\begin{bmatrix} A+B &A+B \\ A+B & A+B+C \end{bmatrix} \geq \begin{bmatrix} A & A \\ A& A+C \end{bmatrix}.
\end{eqnarray*}
Since $\Phi$ is in class $S_{{\rm mon} +}^{(2)}$, we have
\begin{eqnarray*}
\begin{bmatrix} \Phi(A+B) &\Phi(A+B) \\ \Phi(A+B) & \Phi(A+B+C) \end{bmatrix} \geq \begin{bmatrix} \Phi(A) & \Phi(A) \\ \Phi(A)& \Phi(A+C) \end{bmatrix},
\end{eqnarray*}
which implies that
\begin{eqnarray*}
\begin{bmatrix} \Phi(A+B) -\Phi(A) &\Phi(A+B)-\Phi(A) \\ \Phi(A+B)-\Phi(A) & \Phi(A+B+C)-\Phi(A+C) \end{bmatrix} \geq 0.
\end{eqnarray*}
Therefore, by using (\ref{positivee}), we get
\begin{eqnarray}\label{supchoi}
\Phi(A+B+C)-\Phi(A+C) \geq \Phi(A+B) -\Phi(A),
\end{eqnarray}
whence
\begin{eqnarray}\label{super222}
\Phi(A+B+C) +\Phi(A) \geq \Phi(A+B)+\Phi(A+C)
\end{eqnarray}
as required.

Inequality (\ref{supchoi}) shows that if $\Phi$ is in the class $S_{{\rm mon} +}^{(2)}$, then inequality (\ref{alisup}) turns into
\begin{eqnarray}\label{supalii}
\Phi(0)+ \Phi(A+B)\geq \Phi(A)+\Phi(B).
\end{eqnarray}
We give an example to show that inequality (\ref{supalii}) is sharp. Let $\mathscr{A}, \mathscr{B}$ be $C^*$-algebras and $C\in \mathscr{B}$ be a positive element. Consider positive map $\Phi: \mathscr{A} \to \mathscr{B}$ by $\Phi(A)= C$. Clearly $\Phi$ is a completely positive map and
\begin{eqnarray}
\Phi(0)+\Phi(A+B)= \Phi(A)+\Phi(B).
\end{eqnarray}
Next, we show that inequality (\ref{supchoi}) does not hold, in general if we reduce our assumption to the $2$-positivity of $\Phi$.
\begin{example}
Let $C[0,1]$ be the commutative $C^*$-algebra of complex-valued functions on $[0,1]$. It follows from Corollary \ref{norm2} that the $C^*$-norm of $C[0,1]$ is $2$-positive. Taking $ f(x)= 0, \ g(x)= x$ and $h(x)= 1-x$, we observe that
\begin{eqnarray*}
\| f\| + \| f+g+h \| = 1 < 2= \| f+g\| +\| f +h\|.
\end{eqnarray*}
Hence, inequality (\ref{supchoi}) is not valid for $\| \cdot \|: C[0,1] \to \mathbb{C}$ as a $2$-positive map.
\end{example}

The next theorem gives a result about the starshapeness of some classes of positive maps. The theorem can be obtained even if we only assume that the continuous map $\Phi:\mathscr{A}\to \mathscr{B}$ is strongly superadditive map on any (closed under addition) subset $\mathscr{X}$ of $\mathscr{A}$ with $\Phi(0)=0$. Hence, the strong superadditivity of a continuous map $\Phi$ (with $\Phi(0)=0$) on a (closed under addition) subset $\mathscr{X}$ of $\mathscr{A}$ implies the starshapeness of $\Phi$ on $\mathscr{X}$.
\begin{theorem}\label{lineali}
Let $\mathscr{A}, \mathscr{B}$ be two $C^*$-algebras. If $\Phi:\mathscr{A}\to \mathscr{B}$ is a continuous map in the class $S_{{\rm mon} +}^{(2)}$, then the following statements are equivalent:
\begin{enumerate}
\item $\Phi(0)=0$,
\item $\Phi(\alpha A)\leq \alpha \Phi(A) {\text{\ \ for every } } \alpha \in (0,1] \text{\ and\ } A\in \mathscr{A}_+$,
\item $\Phi(\alpha A)\geq \alpha \Phi(A) {\text{\ \ for every } } \alpha \in [1,\infty) \text{\ and\ } A\in \mathscr{A}_+$.
\end{enumerate}
\end{theorem}
\begin{proof}
 (1) $\Longrightarrow$ (2).
Let $\alpha \in [0,1]$ and $A\in \mathscr{A}_+$. Set $\Gamma:=\{ \alpha \in [0,1] \mid \Phi(\alpha A) \leq \alpha \Phi(A)\}$. Trivially $\{0,1\} \subset \Gamma $ and $\Gamma$ is closed. Since $[0,1]$ is convex, to see $\Gamma =[0,1]$ it is enough to show that $\beta, \gamma \in \Gamma$ implies that $\dfrac{\beta +\gamma}{2} \in \Gamma $.\\
Without loss of generality, we may assume that $\beta \leq \gamma$. Since $\Phi$ is in the class $S_{{\rm mon} +}^{(2)}$, it follows from inequality (\ref{super222}) with $\beta A$ instead of $A$ and $\dfrac{\gamma-\beta}{2}A$ instead of both $B$ and $C$ that
\begin{eqnarray*}
 \Phi(\gamma A) +\Phi(\beta A) \geq \Phi\left(\dfrac{\beta +\gamma }{2}A\right)+\Phi\left(\dfrac{\beta +\gamma }{2}A\right).
\end{eqnarray*}
Moreover, since $\beta, \gamma \in \Gamma$, we get
\begin{eqnarray*}
\gamma\Phi( A) +\beta \Phi(A) \geq 2 \Phi\left(\dfrac{\beta +\gamma }{2}A\right),
\end{eqnarray*}
which ensures that $\dfrac{\beta +\gamma}{2} \in \Gamma $.\\
 (2) $\Longleftrightarrow$ (3).
It is enough replace $\alpha$ by $\dfrac{1}{\alpha}$.\\
(2) $\Longrightarrow$ (1). We have $\Phi(\dfrac{1}{n}A) \leq \dfrac{1}{n}\Phi(A)$. Letting $n \longrightarrow \infty$ and using the continuity of $\Phi$ we get (1).
\end{proof}

The following lemma and Theorem \ref{mainnn} give the homogeneity and the linearity of some positive maps belonging to certain classes of maps.
We denote the closed unit ball of $\mathbb{C}$ by $\mathbb{C}_1$.
\begin{lemma}\label{mainn}
Let $\mathscr{A}, \mathscr{B}$ be two unital $C^*$-algebras, $\Phi:\mathscr{A}\to \mathscr{B}$ be a continuous unital $3$-positive map in the class $S_{{\rm mon} +}^{(2)}$ and $\Phi(0)=0$. Then \begin{enumerate}
\item if $\Phi(\alpha I)= \alpha I$ for some $\alpha \in \mathbb{C}_1 \cup \mathbb{R}_+$, then $\Phi(\alpha A) =\alpha \Phi(A)$ for every $A \in \mathscr{A}$,
\item if $\Phi(\alpha I)= \alpha I$ for some $\alpha \in \mathbb{C}_1 \cup \mathbb{R}_+$ with $|\alpha|\neq 0,1 $, then $\Phi(\beta A +B)=\beta \Phi(A) +\Phi(B)$ for every $\beta \in \mathbb{R}_+$ and $A,B \in \mathscr{A}_+$.
\end{enumerate}
\end{lemma}
\begin{proof}
(1) First assume that $\alpha \in \mathbb{C}_1 $. Using Choi's inequality and employing the self-adjointness property of $\Phi$, we get
\begin{eqnarray*}
\Phi(\bar{\alpha} \alpha I ) \geq \Phi(\alpha I)^* \Phi (\alpha I) = \bar{\alpha} \alpha I,
\end{eqnarray*}
whence, by applying Theorem \ref{lineali}, we infer that
\begin{eqnarray*}
\Phi(\bar{\alpha} \alpha I) =\bar{\alpha} \alpha I.
\end{eqnarray*}
 Hence, $\Phi\big((\alpha I)^* \alpha I\big)=\Phi(\alpha I)^* \Phi(\alpha I)$.
 Now, Theorem \ref{3positive} yields
\begin{eqnarray*}
\Phi(\alpha A)= \Phi\big(A(\alpha I)\big)= \Phi(A)\Phi(\alpha I)= \alpha \Phi(A)
\end{eqnarray*}
for every $A\in\mathscr{A}$.\\
In the case when $\alpha \in \mathbb{R}^+ $, we give the proof only for the case $\alpha > 1$; the other case follows form the case when $\alpha \in \mathbb{C}_1$. We first show that $\Phi\left(\dfrac{1}{\alpha}I\right)=\dfrac{1}{\alpha}I$. It follows from Choi's inequality (Corollary \ref{Choi1} (b)) that
\begin{eqnarray*}
\Phi\left(\dfrac{1}{\alpha } I \right) \geq \left(\Phi(\alpha I) \right)^{-1}=\dfrac{1}{\alpha}I,
\end{eqnarray*}
since $\Phi(\alpha I)=\alpha I$. From $\dfrac{1}{\alpha}< 1$ and Theorem \ref{lineali} we infer that
\begin{eqnarray}\label{reveq}
\Phi\left(\dfrac{1}{\alpha}I\right)=\dfrac{1}{\alpha}I.
\end{eqnarray}
Therefore,
\begin{eqnarray*}
\dfrac{1}{\alpha^2}\Phi(I)&\geq& \Phi(\dfrac{1}{\alpha^2}I) \qquad \text{(by Theorem \ref{lineali})}\\ &\geq& \Phi(\dfrac{1}{\alpha}I)\Phi(\dfrac{1}{\alpha}I) \qquad \text{(by Choi's inequality)}\\ &=& \left(\dfrac{1}{\alpha}I\right)^2 \qquad \text{(by (\ref{reveq}).)}
\end{eqnarray*}
Hence,
\begin{eqnarray*}
\Phi\left(\dfrac{1}{\alpha^2} I \right)= \dfrac{1}{\alpha^2}I= \Phi\left(\dfrac{1}{\alpha}I\right)\Phi\left(\dfrac{1}{\alpha}I\right).
\end{eqnarray*}
Now, Theorem \ref{3positive} ensures that
\begin{eqnarray}\label{tas1}
\Phi\left(\dfrac{1}{\alpha} A\right)= \dfrac{1}{\alpha} \Phi(A)
\end{eqnarray}
for every $A\in\mathscr{A}$. Replacing $A$ by $\alpha A$ in equality (\ref{tas1}), we get $\Phi({\alpha} A)= {\alpha} \Phi(A)$.

(2) Set $\Omega := \{ \alpha \in \mathbb{C}_1 \cup \mathbb{R}_+\ |\ \Phi(\alpha I)=\alpha I\}$. From the previous part we have $\Omega= \{ \alpha \in \mathbb{C}_1 \cup \mathbb{R}_+\ |\ \Phi(\alpha A)=\alpha \Phi(A),\ \text{for every} \ A \in \mathscr{A}\}$. First note that\\
\textbf{Fact(1)}: If $0\neq\alpha\in \mathbb{R}_+$ and $ \alpha \in \Omega$, then $ \dfrac{1}{\alpha}$ is in $\Omega$, since $\Phi\Big(\alpha \dfrac{A}{\alpha}\Big)=\alpha \Phi\Big(\dfrac{A}{\alpha}\Big)$ and so $\Phi\Big(\dfrac{1}{\alpha}A\Big)=\dfrac{1}{\alpha}\Phi(A)$.\\
 \textbf{Fact(2)} : For every $\alpha \in \Omega$ we have
\begin{eqnarray*}
\Phi(|\alpha|^2A)&=&\Phi(\alpha \bar{\alpha}A)= \alpha \Phi(\bar{\alpha}A)=\alpha\Phi((\alpha A^*)^*)\\ &=& \alpha\Phi(\alpha A^*)^*=\alpha \bar{\alpha}\Phi(A^*)^*= \alpha \bar{\alpha}\Phi(A)= |\alpha|^2 \Phi(A).
\end{eqnarray*}
 By induction, we get $|\alpha|^{2n}, |\alpha|^{-2n}$ ($ n\in \mathbb{N}$) are in $\Omega$.\\
Next, let $\alpha \in \Omega$ such that $|\alpha|\neq 0,1,$ and that $\beta \in \mathbb{R}_+$. We aim to show that $\beta \in \Omega$. Using Fact(1), we may assume that $\beta \geq 1$. In the case when $1\leq\beta \leq |\alpha|$, we have
\begin{eqnarray*}
\Phi\left(\beta I\right)&=& \Phi\left( |\alpha|^2\dfrac{\beta}{|\alpha|^2}I\right)\\ &=&|\alpha|^2 \Phi\left( \dfrac{\beta}{|\alpha|^2}I\right) \qquad \text{ (since $|\alpha|^2 \in \Omega$)} \\ &\leq& |\alpha|^2 \dfrac{\beta}{ |\alpha|^2}\Phi(I) \qquad \text{ (by Theorem \ref{lineali})} \\ &=&\beta I.
\end{eqnarray*}
 Since $\beta\geq 1$, Theorem \ref{lineali} implies that $\Phi\left(\beta I\right)=\beta I$.
 In the case when $ 1 <|\alpha| \leq \beta$ (resp. $ 0<|\alpha| <1\leq \beta$), there is a number $n \in \mathbb{N}$ such that $1\leq\beta\leq |\alpha|^{2n}$ (resp. $1\leq\beta\leq \dfrac{1}{|\alpha|^{2n}}$). Now, the same reasoning as above and Fact(2) show that $\Phi(\beta I)=\beta I$. Thus, in all cases, we have $\beta \in \Omega$. \\
We proceed to show that $\Phi$ is additive on positive elements of $\mathscr{A}$. Let $A$ and $B$ be in $\mathscr{A}_+$ and $A\leq B$.
It follows from the previous part of the proof that $\Phi(2C)= 2\Phi(C)$ for every $C\in \mathscr{A}$. Since $\Phi$ is a strongly superadditive map on positive elements, we get
\begin{eqnarray*}
\Phi(A)+\Phi(A+B+(B-A))\geq \Phi(A+B)+\Phi(A+(B-A)),
\end{eqnarray*}
which implies that $\Phi(A)+\Phi(B)\geq \Phi(A+B)$. Since $\Phi$ is superadditive on $\mathscr{A}_+$ and $\Phi(0)=0$, we conclude that $\Phi(A+B)=\Phi(A) +\Phi(B)$.
For arbitrary positive elements $A,B$, without loss of generality, we can assume that $B\neq 0$. There exists $n\in \mathbb{N}$ such that $A \leq B+ n\|B\|I$. Hence,
\begin{align*}
\Phi(A+B) + \Phi(n\|B\|I) &\leq \Phi(A+B +n\|B\| I) \qquad \text{(since $\Phi$ is superadditive) } \\ &= \Phi(A)+\Phi(B+n\|B\| I) \qquad \text{(since $A \leq B+n\|B\| I$) }\\ &= \Phi(A)+\Phi(B)+\Phi(n\|B\| I) \qquad \text{ (Since $B \leq n \|B\| I$)},
\end{align*}
which shows that $\Phi(A+B)\leq \Phi(A)+\Phi(B)$. Consequently, $\Phi(A+B)=\Phi(A)+\Phi(B)$ for every $A,B\in \mathscr{A}_+$.
\end{proof}
Now, we are ready to establish our next result.
\begin{theorem}\label{mainnn}
Let $\mathscr{A}, \mathscr{B}$ be two unital $C^*$-algebras, and $\Phi:\mathscr{A}\to \mathscr{B}$ be a continuous unital positive map and $\Phi(0)
=0$. If either $\Phi$ is $6$-positive or in the class $S_{{\rm mon} +}^{(4)}$, then
 \begin{enumerate}
\item if $\Phi(\alpha I)= \alpha I$ for some $ \alpha \in \mathbb{C}_1 \cup \mathbb{R}_+$ with $|\alpha|\neq 0,1 $, then $\Phi(\beta A +B)=\beta \Phi(A) +\Phi(B)$ for every $\beta \in \mathbb{R}$ and $A,B \in \mathscr{A}$,
\item if $\Phi\big((r+\mbox{i}s) I\big)= (r+\mbox{i}s) I$ for some $r,s \in \mathbb{R}$ with $0\neq |r+\mbox{i}s|<1$ and $s\neq 0$, then $\Phi$ is linear on $\mathscr{A}$.
\end{enumerate}
\end{theorem}
\begin{proof}

(1) Let $\Phi$ be $6$-positive. We begin by proving that $\Phi$ is additive on $\mathscr{A}$. The map $\Phi_2:M_2(\mathscr{A})\to M_2(\mathscr{B})$ is $3$-positive and $\Phi_2(0)=0$. Thus, using Theorem \ref{3positive2}, $\Phi_2$ is superadditive on positive elements of $M_2(\mathscr{A})$. Let $\tilde{A}=\begin{bmatrix}
 \|A\|I & A^* \\ A & \|A\|I
\end{bmatrix}$ and $\tilde{B}=\begin{bmatrix}
 \|B\| I & B^* \\ B & \|B\| I
\end{bmatrix}$. Trivially, $\tilde{A}, \tilde{B}\in M_2(\mathscr{A})_+$. It follows from the superadditivity property of $\Phi_2$ on positive elements that
\begin{align*}
&\hspace{-1cm}\begin{bmatrix}
 \Phi(\|A\|I+\|B\| I) & \Phi(A^*+B^*) \\ \Phi(A+B) & \Phi(\|A\|I+\|B\| I)
\end{bmatrix}
\\ &=\Phi_2(\tilde{A}+\tilde{B})
\\ &\geq \Phi_2(\tilde{A})+\Phi_2(\tilde{B})
\\ &= \begin{bmatrix}
 \Phi(\|A\|I) & \Phi(A)^* \\ \Phi(A) & \Phi(\|A\|I)
\end{bmatrix} + \begin{bmatrix}
 \Phi(\|B\| I) & \Phi(B)^* \\ \Phi(B) & \Phi(\|B\| I)
\end{bmatrix}.
\end{align*}
Therefore, by using the second part of Lemma \ref{mainn}, we get
\begingroup\makeatletter\def\f@size{11}\check@mathfonts
\begin{align*} &
\begin{bmatrix}
 \Phi(\|A\|I+\|B\| I) - \Phi(\|A\|I)-\Phi(\|B\| I)& \Phi(A^*+B^*) - \Phi(A^*)-\Phi(B^*) \\ \Phi(A+B) - \Phi(A)-\Phi(B) & \Phi(\|A\|I+\|B\| I)- \Phi(\|A\|I)-\Phi(\|B\| I)
\end{bmatrix}\\ & = \begin{bmatrix}
 0& \Phi(A^*+B^*) - \Phi(A^*)-\Phi(B^*) \\ \Phi(A+B) - \Phi(A)-\Phi(B) & 0
\end{bmatrix} \geq 0,
\end{align*}
\endgroup
whence $\Phi(A+B) =\Phi(A)+\Phi(B)$ for every $A,B \in \mathscr{A}$.\\
The task is now to show $\Phi(\beta A)=\beta \Phi(A)$ for every $\beta \in \mathbb{R}$ and $A\in \mathscr{A}$. Part (2) of Lemma \ref{mainn} ensures $\Phi(\beta I)=\beta I$, for all $\beta \in \mathbb{R}_+$ and so part (1) then shows $\Phi(\beta A)=\beta \Phi(A)$ for all $A\in \mathscr{A}$ and all $\beta \in \mathbb{R}_+$. Let $\beta \in \mathbb{R}_-$. The additivity of $\Phi$ implies that
\begin{eqnarray}\label{msm5}
\Phi(-A)=-\Phi(A)
\end{eqnarray}
for every $A\in \mathscr{A}$. Hence,
\begin{eqnarray*}
\Phi(\beta A)&=&\Phi\big(-(-\beta A)\big)\\ &=& -\Phi(-\beta A) \qquad \text{(by (\ref{msm5})})\\ &=&\beta \Phi(A) \qquad \text{(since $-\beta \in \mathbb{R}_+$.)}
\end{eqnarray*} This implies that $\Phi(\beta A+ B)= \beta \Phi(A) +\Phi(B)$ for every $\beta \in \mathbb{R}$ and $A, B \in \mathscr{A}$. \\

(2) To prove the linearity of $\Phi$, we only need to prove that $\Phi(\mbox{i} I)=\mbox{i} I$.\\ Using the equality $\Phi\big((r+\mbox{i}s) I\big)= (r+\mbox{i}s) I$, we have
\begin{eqnarray*}
rI+\mbox{i}sI&=& \Phi(rI+\mbox{i}sI) \\ &=& \Phi(rI)+ \Phi(s\mbox{i}I)\qquad \text{(since $\Phi$ is additive)}\\ &=& r\Phi(I)+s\Phi(\mbox{i} I) \qquad \text{ (since $\Phi$ is homogeneous on $\mathbb{R}$)}\\ &=& rI+s\Phi(\mbox{i} I),
\end{eqnarray*}
which implies that $\Phi(\mbox{i}I)= \mbox{i} I$.\\
A similar argument can be used to show the results in the case when $\Phi$ is in the class $S_{{\rm mon} +}^{(4)}$. Indeed, if $\Phi$ is in the class $S_{{\rm mon} +}^{(4)}$, then $\Phi_2$ is monotone on positive elements of $M_2(\mathscr{A})$ as well as $\Phi$ is $4$-positive, since $\Phi_4$ is monotone on positive matrices in $M_4(\mathscr{A})$. Hence, $\Phi_2$ is in the class $S_{{\rm mon} +}^{(2)}$. Therefore, inequality \eqref{supalii} yields that $\Phi_2$ is superadditive on the positive elements of $M_2(\mathscr{A})$. The result follows by a similar method as in the case when $\Phi$ is $6$-positive.
\end{proof}

 Lemma \ref{mainn} and the above theorem are valid even if we replace the conditions $\Phi(I)=I$ and $\Phi({\alpha I})=\alpha I$ by $\Phi(I)$ is invertible and $\Phi(\alpha I)=\alpha\Phi(I)$, respectively. Indeed, let $\Phi: \mathscr{A}\to \mathscr{B} $ be a map between $C^*$-algebras which is in the class $S_{{\rm mon} +}^{(n)}$. Then the positivity of the Schur product of positive elements yield that the map $\Psi :\mathscr{A}\to \mathscr{B}$ given by $\Psi(X)=\Phi(I)^{-\frac{1}{2}} \Phi(X) \Phi(I)^{-\frac{1}{2}}$ is a map in the class $S_{{\rm mon} +}^{(n)}$, which is unital. However, the condition $\Phi(0)=0$ can not be admitted, in general (see Example \ref{exfin} (1)). \\

Finally, we give some examples to show the necessity of some hypotheses in Lemma \ref{mainn} and Theorem \ref{mainnn}.

\begin{example}\label{exfin}
(1) Consider the map $\varphi : \mathbb{C}\to \mathbb{C}$ defined by $\varphi(z)=\dfrac{1}{3}(|z|^{\frac{3}{2}} +|z|^{\frac{4}{3}} +1)$. It is known that $\varphi$ is a $3$-positive map (see \cite[Theorem 6.3.9]{horn}) and $\varphi \in S_{{\rm mon} +}^{(2)}$ (see \cite[Theorem 5.1]{hiai}). Evidently, there exists a number $z \in [1.1, 2]$ such that $\varphi(z) = z$. However, $\varphi$ is not additive on positive numbers. This shows that the condition $\Phi(0)=0$ cannot be removed in Lemma \ref{mainn}.\\
(2) Let $(\mathscr{A},\| \cdot\|)$ be a unital $C^*$-algebra. Theorem \ref{3positive2} ensures that $\| \cdot \|$ is not a $3$-positive map in the most $C^*$-algebras. For every $\alpha >0$, we have $\|\alpha I\|=\alpha$ while $\| \cdot \|$ is not additive on positive elements of $\mathscr{A}$, in general. It shows that $3$-positivity is a necessary condition in Lemma \ref{mainn}. Thus we arrived at another interesting question whether there is a map in the class $S_{{\rm mon} +}^{(2)}$ which is not $3$-positive.\\
(3) The map $|\cdot |$ is a $3$-positive norm on $\mathbb{C}$; see \cite[Section 7.5 Problem 4]{horn2}. In addition $|\cdot|$ is in the class $S_{{\rm mon} +}^{(2)}$.
Indeed, if $\begin{bmatrix}
a &b \\ b^* &c
\end{bmatrix}\geq \begin{bmatrix}
e &f \\ f^* &g
\end{bmatrix}\geq 0
$, then we have
\begin{eqnarray*}
(a-e)(c-g)\geq (b-f)(b^*-f^*)= |b|^2- 2{\rm Re}(f^*b) +|f|^2.
\end{eqnarray*}
Since ${\rm Re}(f^*b)\leq |fb|$, we can write
\begin{eqnarray*}
(|a|-|e|)(|c|-|g|)&=& (a-e)(c-g)\\ &\geq& (b-f)(b^*-f^*)\\ &=& |b|^2- 2{\rm Re}(f^*b) +|f|^2\\ &\geq& |b|^2- 2|f||b| +|f|^2 \\ &=&(|b|-|f|)^2.
\end{eqnarray*}
This immediately implies that
\begin{eqnarray*}
\begin{bmatrix}
|a| &|b| \\ |b^*| &|c|
\end{bmatrix}\geq \begin{bmatrix}
|e| &|f| \\ |f^*| &|g|\end{bmatrix}.
\end{eqnarray*}
However, $|\cdot|$ is not $6$-positive, nor is in the class $S_{{\rm mon} +}^{(4)}$. For every $\alpha >0$, we have $|\alpha I|=\alpha$, but $|\cdot|$ is not additive on $\mathbb{C}$. This shows that Theorem \ref{mainnn} is not valid for $3$-positive maps in the class $S_{{\rm mon} +}^{(2)}$, in general.
\end{example}

\noindent\textbf{Acknowledgement}\\
The authors would like to sincerely thank the anonymous referee for carefully reading the paper and for very useful comments.  

\end{document}